\documentclass[11pt]{article}
\usepackage{graphicx}
\usepackage{tikz}
\usepackage{amssymb}
\usepackage{amsmath,amsthm}
\usepackage{verbatim}
\usepackage{amsfonts}
\usepackage[T1]{fontenc}
\usepackage{multicol}
\usepackage{color}

\newtheorem{thm}{Theorem}
\newtheorem{prop}{Proposition}[section]

\newtheorem{lem}[thm]{Lemma}
\newtheorem{conj}[thm]{Conjecture}

\newtheorem{rem}[prop]{Remark}

\newcommand{\List}{\mathcal{L}}

\title{The List Distinguishing Number of Kneser Graphs}
\author{Niranjan Balachandran\footnote{niranj (at) math.iitb.ac.in, Supported by grant 12IRCCSG016, IRCC, IIT Bombay} and Sajith Padinhatteeri\footnote{sajith(at)math.iitb.ac.in, Supported by grant 09/087(0674)/2011-EMR-I, Council of Scientific \& Industrial Research, India}, \\ Department of Mathematics,\\
Indian Institute of Technology Bombay,\\ Mumbai, India. }

\begin{document}
\maketitle
\begin{abstract} A graph $G$ is said to be $k$-distinguishable if the vertex set can be colored using $k$ colors such that no non-trivial automorphism fixes every color class, and the distinguishing number $D(G)$ is the least integer $k$ for which $G$ is $k$-distinguishable. If for each $v\in V(G)$ we have a list $L(v)$ of colors, and we stipulate that the color assigned to vertex $v$ comes from its list $L(v)$ then $G$ is said to be $\List$-distinguishable where $\List =\{L(v)\}_{v\in V(G)}$. The list distinguishing number of a graph, denoted $D_l(G)$, is the minimum integer $k$ such that every collection of lists $\List$ with $|L(v)|=k$ admits an $\List$-distinguishing coloring. In this paper, we prove that $D_l(G)=D(G)$ when $G$ is a Kneser graph.\end{abstract}


2010 AMS Classification Code: 05C15, 05C25, 05C76, 05C80.

\section{Introduction}
Let $G$ be a graph and let $Aut(G)$ denote the full automorphism group of $G$. By  an $r-$vertex labelling of $G$, we shall mean a map $f:V(G) \rightarrow \{1,2, \dots, r\}$, and the sets $f^{-1}(i)$ for $i\in\{1,2\ldots,r\}$ shall be referred to as the color classes of the labelling $f$. An automorphism $\sigma \in Aut(G)$ is said to fix a color class $C$ of $f$ if $\sigma(C) = C$, where $\sigma(C) = \{\sigma(v):v \in C\}$. Albertson and Collins \cite{AK} defined the distinguishing number of graph $G$, denoted $D(G)$, as the minimum $r$ such that $G$ admits an $r-$ vertex labelling with the  property that no non-trivial automorphism of $G$ fixes every color class. 
 
An interesting variant of the distinguishing number of a graph, due to Ferrara, Flesch, and Gethner \cite{MF} goes as follows. Given an assignment $\List=\{ L(v) \}_{v \in V(G)}$ of lists of available colors to vertices of $G$, we say that $G$ is $\List-$distinguishable if there is a distinguishing coloring $f$ of $G$ such that $f(v) \in L(v)$ for all $v$. The list distinguishing number of $G$, $D_l(G)$ is the minimum integer $k$ such that $G$ is $L-$distinguishable for any list assignment $L$ with $|L(v)|=k$ for all $v$. The list distinguishing number has generated a bit of interest recently (see \cite{MF, tree, interval} for some relevant results) primarily due to the following conjecture that appears in \cite{MF}: For any graph $G$, $D_l(G)=D(G)$. The paper \cite{MF}, in which this notion was introduced and the conjecture was made, proves the same for cycles of size at least $6$, cartesian products of cycle, and for graphs whose automorphism group is the Dihedral group. The paper \cite{tree} proves the validity of the conjecture for trees, and \cite{interval} establishes it for interval graphs. 

Let $r\ge 2$, and $n\ge 2r+1$. The  Kneser graph $K(n,r)$ is defined as follows: The vertex set of $K(n,r)$ consists of all $k-$element subsets of $[n]$; vertices $u,v$ in $K(n,r)$ are adjacent if and only if $u\cap v=\emptyset$. The Distinguishing number of the Kneser graphs is well known (see \cite{AB}): $D(K(n,r))=2$  when $n \neq 2r+1$ and $r\ge 3$; for $r=2$, $D(K(5,2))=3$, and $D(K(n,2))=2$ for all $n\ge 6$. 

Our main theorem in this short paper establishes the validity of the list distinguishing conjecture for the family of Kneser graphs. 
\begin{thm} $D_l(K(n,r))=D(K(n,r))$ for all $r\ge 2, n\ge 2r+1$.\end{thm}
Before we proceed to the proof of the theorem, we describe the main idea of the proof. We choose randomly (uniformly) and independently for each vertex $v$, a color from its list $L(v)$, and we calculate/bound the expected number of non-trivial automorphisms that fix every color class for this random set of choices. This line of argument features in some other related contexts, for e.g., \cite{BP,contuck,IRL,RS} most notably under the umbrella of what is called the `Motion Lemma', and some of its variants. These methods however do not work  in the cases  $r=2$ and $n=6$ or $n=7$, so we settle these cases by different arguments. As it turns out, the case with $r\ge 3$ is much simpler in contrast to the case $r=2$.

The rest of the paper is organized as follows. In the next couple of sections, we detail the proof of the conjecture for $r=2$ with $n\ge 8$, and the case for $r\ge 3$ respectively. The cases $r=2, n=6,7$  are dealt with in the appendix. We conclude with a few remarks.
\section{List-distinguishing number of $K(n,2)$}\label{r=2}
Firstly, recall the following
\begin{thm} $D(K(n,2)) = 2$ for $n\ge 6$, and $D(K(5,2))=3$.\end{thm} 
Since $n\ge 2r+1$, it follows from the Erd\H{o}s-Ko-Rado theorem that the full automorphism group of $K(n,r)$ is precisely $S_n$, the symmetry group on $n$ symbols. 

Suppose $n\ge 6$ and suppose $\{L(v)\}_{v \in V(K(n,2))}$ is a collection of color lists of size $2$ corresponding to the vertices of $K(n,2)$. It is simpler to think of these as color lists on the edges of $K_n$.  In other words, let $\{L(e)\}_{e \in E(K_n)}$ be lists of colors of size $2$ for the edges of $K_n$.  For each edge of $K_n$ we choose a color uniformly and independently at random from its given list of colors.  As mentioned in the introduction, we seek to compute the expected number of non-trivial automorphisms that fix all the colors class of this random coloring. 
 
 Firstly we set up some notation.
\begin{itemize}
\item[a.] If the disjoint cycle decomposition of a permutation $\sigma\in S_n$ consists of  $l_i$ cycles of length $\lambda_i,$ for  $i=1,2,\ldots,t$ with $\lambda_1<\lambda_2<\cdots<\lambda_t$, then we say $\sigma$ is of type $\Lambda$ where $\Lambda := ( \lambda_1^{l_1}, \lambda_2^{l_2}, \dots, \lambda_t^{l_t})$. Note that $\sum_i l_i\lambda_i = n$.
\item[b.] $CT^{(n)}$ shall denote the set of all permutation types of permutations in $S_n$, i.e., $$CT^{(n)}:=\{ ( \lambda_1^{l_1}, \lambda_2^{l_2}, \dots, \lambda_t^{l_t})\textrm{\ with\ }\sum_i l_i\lambda_i = n\textrm{\ and\ }\lambda_1<\lambda_2<\cdots<\lambda_t\}.$$   
\item[c.] $CT^{(n)}_{\ge r}$, $CT^{(n)}_{r}$ shall denote the sets of all cycle types with minimum cycle length  at least $r$, and with minimum cycle length exactly $r$, respectively.
\item[d.]For positive integers $a,b$, we shall denote by $(a,b)$, the g.c.d. of $a$ and $b$.
\end{itemize}
We start with two simple observations. Firstly, note that if a nontrivial automorphism $\sigma$ fixes each of the color classes (as sets) in the random coloring of  $E(K_n),$ then for each $e \in E(K_n)$ all the edges in the orbit of $e$ under the action of $\sigma$ get the same color. Also, the probability that $\sigma$ preserves every color class depends only on the cycle type of $\sigma$.

For a nontrivial $\sigma\in S_n$, we first obtain an upper bound $P(\sigma)$ on the probability that $\sigma$ fixes all the color classes (as sets) in the random coloring. We set $P(\Lambda) := \displaystyle\sum_{\sigma\textrm{\ of\ type\ }\Lambda} P(\sigma)$, so this gives an upper bound on $P(\Lambda)$ as well.

 \begin{lem}\label{lemcore}
Let $\sigma\in S_n$ be a nontrivial permutation of type $\Lambda = ( \lambda_1^{l_1}, \lambda_2^{l_2}, \dots, \lambda_t^{l_t})$.  Let $g(\lambda_i) :=
\left \lfloor \frac{(\lambda_i -1)^2}{2}\right \rfloor$ and $g(\lambda_i, \lambda_j ):= \lambda_i \lambda_j - ( \lambda_1, \lambda_j)$. Furthermore, for $i\le j$ write $l_j^* := l_i(l_i-1) /2$ when $i=j$ and $l_j^* = l_il_j$ for all $j>i$. Then
 $$P(\sigma) = \frac{1}{2^{\mu}},$$ where $$\mu = \sum\limits_{i=1}^t \left(g(\lambda_i)l_i +\sum\limits_{j\ge i}^t g(\lambda_i, \lambda_j )l_j^* \right).$$ Consequently, for $\Lambda \in CT^{(n)},$
 $$P(\Lambda) \le n!2^{-\mu} \prod_{i=1}^t \frac{\lambda_i^{-l_i}}{(l_i)!}.$$  
\end{lem}
 \begin{proof}
If $\sigma$ is an automorphism that fixes every color class then as observed earlier, for each edge $e$, every edge in $\{e, \sigma(e), \sigma^2(e), \dots, \sigma^k(e) \}$ has the same color. Here, the integer $k\ge 1$ is the smallest integer satisfying $\sigma^{k+1}(e) = e.$ Let $\sigma=C_1C_2\cdots C_u$ be its disjoint cycle decomposition. Writing $C=(12\cdots r)$, the prior observation implies that for each $1\le i\le \left \lfloor r/2\right \rfloor$, the set of edges $\{(1,i), (2,i+1),\ldots, (r,i+r-1)\}$ is monochrome, where the addition is performed modulo $r$. Moreover, as these form a pairwise disjoint partition of the edges of the clique on $C$, the probability that all these sets of edges are monochrome is at most $2^{-g(r)}$ where $g(r) = \left \lfloor\frac{(r -1)^2}{2}\right \rfloor$ as in the statement of the lemma.
 
Now, by a similar argument, if $\sigma$ fixes every color class then the set of edges between the vertices of two disjoint cycles $C_i$ and $C_j$ of size $r, s$ respectively is partitioned into monochrome sets of size equal to the least common multiple of  $r,s$. Hence the probability that such an event occurs is $2^{-{g(r,s)}}$ with $g(r,s)=rs-(r,s)$ as in the statement of the lemma. Moreover, these events (i.e., partitioning of the edges within each cycle $C_i$ and also across a pair of cycles $C_i,C_j$) are pairwise independent, and since $\sigma$ is of type $\Lambda$, it follows that the probability that $\sigma$ fixes every color class is at most $2^{-\mu}$, where $\mu$ is as described in the statement of the lemma.
  
As for $P(\Lambda),$ we use the first part of the lemma in conjunction with the observation that there are $n!\prod_{i=i}^t \frac{\lambda_i^{-l_i}}{(l_i)!}$ permutations of type $\Lambda$ in $S_n.$ 
 \end{proof}
 
 \vspace{0.2cm}
  \begin{rem}\label{probzero}
 If $\sigma \in S_n$ is of  type $\Lambda$, then the bound in the preceding lemma occurs if and only if all the lists are identical. If in fact, for some $i$, the lists for the edges $\{(1,i), (2,i+1),\ldots, (r,i+r-1)\}$ has empty intersection, then $P(\sigma)$ is zero. A similar remark about the lists of edges between the vertices of disjoint cycles $C_i,C_j$ also holds.
 \end{rem}
 If $\Lambda, \Gamma$ are cycle types in $CT^{(n)}$ and $CT^{(n-\lambda_1)}$ respectively, we say that $\Lambda$ `extends' $\Gamma$ if $$\Lambda = ( \lambda_1^{l_1}, \lambda_2^{l_2}, \dots, \lambda_t^{l_t})\textrm{\ and\ } \Gamma = ( \lambda_1^{l_1-1}, \lambda_2^{l_2}, \dots, \lambda_t^{l_t}).$$
 Note that $$P(\Lambda) = R_{\lambda_1}(\Lambda) P( \Gamma)$$ where $$R_{\lambda_1}(\Lambda) = \frac{n(n-1)(n-2) \dots (n-\lambda_1+1)}{\lambda_1 l_1 }2^{-g(\lambda_1) - g(\lambda_1, \lambda_1 )(l_1-1) - \sum\limits_{j\ge 2}^t g(\lambda_1, \lambda_j )l_j}.$$ This is a straightforward consequence of lemma \ref{lemcore}.

\begin{lem}\label{r<1}
Let $\Lambda = ( \lambda_1^{l_1}, \lambda_2^{l_2}, \dots, \lambda_t^{l_t})$ be a cycle type in $CT^{(n)}$. Then for $n\ge 9$
\begin{eqnarray} R_{\lambda_1}(\Lambda) &<&2^{-n \lambda_1/7}\textrm{\ if\ }\lambda_1\ge 2,\\
                           R_{1}(\Lambda) &\le &\frac{n}{2(n-2)}\textrm{\ and\ equality\ is\ achieved\ precisely\ if\ }\Lambda=(1^{n-2},2),\\
                           R_{1}(\Lambda) &\le &\frac{n}{4(n-3)}\textrm{\ if\ }\Lambda\neq(1^{n-2},2).\end{eqnarray}\end{lem}
\begin{proof}

Firstly, suppose $\lambda_1\ge 2$. Set $s = \frac{ \lambda_1( 2n+ 7 \lambda_1)}{14}$; observe that 
\begin{align*}
\log n<\frac{5n}{14}=\frac{n}{2}- \frac{s}{\lambda_1} - \frac{\lambda_1}{2} 
\end{align*}
holds for $n\ge 9$. Here $\log$ is the logarithm to the base $2$. 
As $\lambda_1\ge 2$,  we have $\sum_{i\ge 1} l_i\le n/2$, so we may write 
\begin{align*}
& \frac{n}{2}- \frac{s}{\lambda_1} - \frac{\lambda_1}{2}< n - \sum\limits_{j\ge 1}  l_j  - \frac{s}{\lambda_1} -\frac{\lambda_1}{2} +   \frac{1}{\lambda_1} \left( \log(l_1)+ \log(\lambda_1)\right ).  \nonumber
\end{align*}
Since  $n= \sum_i \lambda_i l_i$ we have (by rearranging the terms)
\begin{align}
& n - \sum\limits_{j\ge 1}  l_j  - \frac{s}{\lambda_1} -\frac{\lambda_1}{2} +   \frac{1}{\lambda_1} \left( \log(l_1)+ \log(\lambda_1)\right ) \label{eq1}\\
&= \sum_i \lambda_i l_i - \left(l_1 + \frac{1}{\lambda_1}\sum\limits_{j\ge 2} \lambda_1 l_j \right)- \frac{s}{\lambda_1} -\frac{\lambda_1}{2} +   \frac{1}{\lambda_1} \left( \log(l_1)+ \log(\lambda_1)\right ) \label{eq2}\\
&= \frac{-s}{\lambda_1}+ \frac{\lambda_1}{2} + l_1 \lambda_1 -\lambda_1 -l_1 +  \sum\limits_{j\ge 2}^t \lambda_jl_j  + \frac{1}{\lambda_1} \left(  \log(l_1)+ 
\log(\lambda_1)  - \sum\limits_{j\ge 2}^t \lambda_1 l_j \right )  \label{eq3}\\
 & =  \frac{-s}{\lambda_1}+ \frac{\lambda_1}{2} + l_1 \lambda_1 -\lambda_1 -l_1 + \frac{1}{\lambda_1} \left( \sum\limits_{j\ge 2}^t (\lambda_1\lambda_j  - \lambda_1)l_j +  \log(l_1)+ \log(\lambda_1) \right ). \label{eq4}
\end{align}
To elaborate, we rewrite $n=\sum_j \lambda_j l_j$ in \eqref{eq1} and write $\sum_j l_j$ as $l_1 + \frac{1}{\lambda_1}\sum\limits_{j\ge 2} \lambda_1 l_j$ to get \eqref{eq2}; \eqref{eq3} results from \eqref{eq2} by rearranging terms, writing $-\frac{\lambda_1}{2}$ as $\frac{\lambda_1}{2}-\lambda_1$ and also isolating the term $\lambda_1 l_1$ from  $\sum_j \lambda_j l_j$, and finally \eqref{eq4} is again a suitable rearrangement of \eqref{eq3}.

 Since $\lambda_1 \ge ( \lambda_1, \lambda_j)$, we have for $n\ge 9$, 
$$ \log n< \frac{-s}{\lambda_1}+ \frac{\lambda_1 -2}{2}  + (\lambda_1 -1)(l_1-1) + \frac{1}{\lambda_1} \left( \sum\limits_{j\ge 2}^t g(\lambda_1, \lambda_j )l_j +  \log(l_1)+ \log(\lambda_1) \right )$$ 
Since $g(x)\le (x^2-2x)/2$ we have 
$$ \lambda_1 \log n < -s+ g(\lambda_1) + g(\lambda_1, \lambda_1 )(l_1-1) + \sum\limits_{j\ge 2}^t g(\lambda_1, \lambda_j )l_j +  \log(l_1)+ \log(\lambda_1)$$
and thus,  $$n^{\lambda_1}< 2^{-s} l_1 \lambda_1  2^{g(\lambda_1) + g(\lambda_1, \lambda_1 )(l_1-1) + \sum\limits_{j\ge 2}^t g(\lambda_1, \lambda_j )l_j}$$ which  achieves the first part of the lemma since $n\lambda_1/7 < \frac{ \lambda_1( 2n+ 7 \lambda_1)}{14}$.

When $\lambda_1=1$, then note that $$R_1(\Lambda)\le\frac{n}{l_1}2^{(\sum\limits_{j\ge 2} (1-\lambda_j)l_j)}=\frac{n}{l_1}2^{L-n},$$ where $L=\sum_{j\ge 1} l_j$.   Since $\lambda_2\ge 2$, it follows that  $n-L\ge (n-l_1)/2$, so we have $$R_1(\Lambda)\le\frac{n}{l_12^{(n-l_1)/2}}. $$
It follows by elementary calculus (for instance) that since the function $h(x)=x2^{(n-x)/2}$ defined on $[1,n-2]$ achieves its minimum value of $2(n-2)$ at $x=n-2$, we have 
$$R_1(\Lambda)\le\frac{n}{2(n-2)}$$ as required. \\
If  $\Lambda$ corresponds to a permutation type of a non-trivial permutation and $\Lambda\ne(1^{n-2},2)$ then arguing as before, we observe that in that case, $n-L\ge 4$, and among such permutation types, $R_1(\Lambda)$ is maximum for $\Lambda=(1^{n-3},3)$, and for this $\Lambda$, $R_1(\Lambda)\le\frac{n}{4(n-3)}$. This completes the proof. \end{proof}

Set $f(n) := \sum\limits_{\sigma \in S_n }P( \sigma)$. Let $f_{\ge i}(n) $ denote the corresponding sum over all those permutations $\sigma\in S_n$ in which every cycle has size at least $i$. Also set $P(n) := P(\Lambda =n).$ 
\begin{lem}\label{f(n)<1}
For any $n$, 
$$ f(n) < \frac{n}{2(n-2)} f(n-1) + \sum\limits_{i=2}^{\left \lfloor n/2\right \rfloor} 2^{-ni/7}f_{\ge i}(n-i) + P(n).$$
\end{lem}
\begin{proof}
 Observe that any cycle type in $CT^{(n)}_{i}$ is an extension of a unique cycle type in $CT^{(n-i)}_{\ge i}$. Also, since for a fixed cycle type $\Lambda \in CT^{(n)},$ there are exactly $N(\Lambda) = n! \prod_{i=1}^t \frac{\lambda_i^{-l_i}}{(l_i)!}$ permutations of  type $\Lambda$, we have 
\begin{align}
f(n) =&\sum\limits_{\Lambda \in CT^{(n)} }N(\Lambda)P(\Lambda) = \sum\limits_{i=1}^{\left \lfloor n/2 \right \rfloor} \sum\limits_{\Lambda \in CT^{(n)}_{i} }N(\Lambda) P(\Lambda) + P(n) \nonumber\\
=& \sum\limits_{i=1}^{\left \lfloor n/2 \right \rfloor} \sum\limits_{\Lambda \in CT^{(n)}_{\ge i} }N(\Lambda) R_i(\Lambda) P(\Lambda) +P(n) \nonumber
\label{eqth2}   
\end{align}
By the bounds for $R_{\lambda_1}(\Lambda)$ from lemma \ref{r<1}, we have  
\begin{align}
f(n) < & \sum\limits_{\Lambda \in CT^{(n-1)} } \frac{n}{2(n-2)} N(\Lambda) P(\Lambda) + \sum\limits_{i=2}^{\left \lfloor n/2 \right \rfloor}\sum\limits_{\Lambda \in CT^{(n)}_{\ge i} }  2^{-in/7} N(\Lambda) P(\Lambda) \nonumber\\ 
<&\  \  \frac{n}{2(n-2)} f(n-1) + \sum\limits_{i=2}^{\left \lfloor n/2\right \rfloor} 2^{-ni/7}f_{\ge i}(n-i) + P(n).
\end{align}
\end{proof}
\begin{thm}
$f(n)<1$ for all $n\ge 8$. In fact, for all $n\ge 8$, $$f(n) \le \frac{kn^2}{2^{n}}$$ for some absolute constant $k$. Consequently, we have $D_l(K(n,2))=2$ for $n\ge 8$.

\end{thm}
\begin{proof}
This proof is by induction on $n.$ It is straightforward, though a little tedious to check $f(8) \approx 0.874 < 1 $ by calculating $\sum\limits_{\Lambda \in CT_8}N(\Lambda)P(\Lambda)$ directly;  we also check that $f_{\ge 4}(5),$ $f_{\ge 3}(6)$ and $f_{\ge 2}(7)$ are strictly less than one. Furthermore,
\begin{align}
f_{\ge 2}(7) =& \frac{7!}{2^2 2! 3}2^{-12} + \frac{7!}{12}2^{-2-4-11} + \frac{7!}{10}2^{-8-9} + \frac{7!}{7}2^{-18} \approx  0.061. \label{7=n>2} \\
f_{\ge 3}(6) =&\frac{6!}{3^2 2!}2^{-2-2-6} + \frac{6!}{6}2^{-12}\approx  0.0683.\label{6=n>3} \\
f_{\ge 4}(5) = &\frac{5!}{5}2^{-8}\nonumber \approx  0.0937.\label{5=n>4}
\end{align}      
Also, 
$P(n) = \frac{n!}{n}2^{-\lfloor \frac{(n-1)^2}{2} \rfloor}.$ Since $P(n)$ is monotonically strictly decreasing for $n\ge 3$, we may bound  $P(n)<P(9) = 8! 2^{-32} \approx 0.0000093.$   

Assume $f(k) < 1$ for $8 \le k \le n-1.$ By lemma \ref{f(n)<1} we have  $$f(n)\le \frac{n}{2(n-2)} f(n-1) + \sum\limits_{r=2}^{\left \lfloor n/2\right \rfloor} 2^{\frac{-nr}{7}}f_{\ge r}(n-r) + P(n)$$ so   
$$
f(n)  \le  \frac{n}{2(n-2)} +\sum\limits_{r=2}^{\lfloor n/2\rfloor} 2^{\frac{-nr}{7}}  + 0.0000093.
$$
 Since $S_n < (2^{n/7}(2^{n/7}-1))^{-1}<0.3$ for $n\ge 9$, we have $f(n) <1$ when $n \ge 9.$

For the exponentially decaying upper bound, we again proceed to do so inductively. The only difference is that this time, we are slightly more careful with our bounds, though we do not attempt to optimize for the constant $k$. We shall show that $f(n)\le 20n^2/2^n$ holds for all $n\ge 8$.

 It is easy to see that this statement holds for $n\le 11$ since $20n^2/2^n$ is greater than $1$ for all these values of $n$. In computing $f(n)$ through the application of lemma \ref{f(n)<1}, we isolate the terms arising from permutations of type $(1^{n-2},2)$ and note that their contribution to the sum $f(n)$ is precisely $n(n-1)/2^{n-1}$. For the remaining $\Lambda$ with $\lambda_1=1$, as observed in lemma \ref{r<1}, we have $R_1(\Lambda)\le n/4(n-3)$. Piecing these together, and by induction, we have 

$$f(n) < \frac{n^2}{2^{n-1}} + \frac{n}{4(n-3)}\frac{20(n-1)^2}{2^{n-1}} + \frac{20n^2}{2^n}\sum\limits_{i\ge 2} \left(\frac{2}{2^{n/7}}\right)^i+ \left(\frac{4n}{2^n}\right)^n.$$ 

Now, our choice of constants gives us that for $n\ge 12$, $\left(\frac{(n-1)^2}{2n(n-3)}+\frac{1}{2^{n/7}(2^{n/7}-2)}\right) < 0.8$, so, the right hand side of the expression above is at most $18n^2/2^n + \left(\frac{4n}{2^n}\right)^n < 19n^2/2^n$, and the induction is complete.
 \end{proof}
\textbf{Remark:} As observed in the proof, $f(n)\ge \displaystyle\frac{\binom{n}{2}}{2^{n-2}}$, so we actually have $f(n)=\Theta(n^2/2^n)$.

\section{List distinguishing number of $K(n,r)$ when $r\ge 3$}\label{r>2}
In this section we show that  $D_l(K(n,r))=2$ for $r\ge 3, n\ge 2r+1$ holds with positive probability for a random coloring, where as before the random coloring is obtained by choosing for each vertex $v$, a color uniformly from its list, and independently across the vertices. Recall that the vertices of $K(n,r)$ correspond to $r$-subsets of $[n]:=\{1,2\ldots,n\}$ and vertices $u,v\in V(K(n,r))$ are adjacent if and only if $u\cap v=\emptyset$. As before, suppose that the vertex $v$ is assigned a color list of size $2$.  
\begin{lem}\label{2cycle}
Consider the random coloring of $G= K(n,r), r \ge 3$. Let $\sigma$ be a nontrivial permutation of type $\Lambda$ that fixes every color class. Then
$$P(\sigma) < \frac{1}{2^m} \textrm{\ where\ } m= {n-2 \choose r-1}.$$
\end{lem}
\begin{proof}
Without loss of generality suppose $\sigma$ has the cycle $(1,2, \hdots, t)$ for some $2\le t \le n.$ Let $v$ be a vertex corresponding to a set containing the element $1$, but not the element $2$ in $[n]$. Then since $2\in\sigma(v)$ it follows that  $v\ne\sigma (v)$. Therefore, if $\sigma$ fixes every color class, since $v$ and $\sigma(v)$ for each vertex $v$ are assigned the same color, each set of size $r$ containing $1$ but not containing $2$ must get mapped to a distinct subset, not of the same form, and each of these pairs of vertices are monochrome pairs. The probability of the aforementioned event is precisely $2^{-m}$ as stated in the lemma.
\end{proof}
\begin{thm}
If $r \ge 3$ and $n > 2r+1,$ then $D_l(K(n,r))= 2.$
\end{thm}
\begin{proof} Consider the random coloring of $K(n,r)$ as described earlier.  By Lemma \ref{2cycle}, the probability that there exists a non-trivial automorphism that fixes every color class under this random coloring is at most $$\frac{|Aut((K(n,r))|}{2^m}=\frac{n!}{2^{\binom{n-2}{r-1}}}\le\frac{n!}{2^{\binom{n-2}{2}}}$$ since $r\ge 3.$ It is straightforward to check that the last expression is less than $1$ for $n \ge 9.$ 

Since $n \ge 2r+1$ and $r\ge3$ the remaining cases are $n=7$ and $n=8.$ In these cases we look at the corresponding expression a little closer.  We bifurcate the set of non trivial automorphisms into two categories: We say a permutation $\sigma\in S_n$ is of Category I if all the cycles in the cycle decomposition of $\sigma$ have size at most $2$, otherwise we say $\sigma$ is a category II permutation.

For $n=7$ there are $\frac{7!}{2.5!} + \frac{7!}{2^2.2!.3!} + \frac{7!}{2^3.3!}  = 231$ nontrivial permutations in Category I and $4808$ permutations in Category II. Let $E_I$ and $E_{II})$ denote the events that a nontrivial automorphism of Category I, Category II respectively, fixes every color class, then
$$P(E) = P(E_I) + P(E_{II}) < \frac{231}{2^{10}} + \frac{4808}{2^{20}} < 1.$$ 
Similarly when $n=8$ $$P(E) < \frac{973}{2^{15}} + \frac{39346}{2^{30}}<1$$ and this completes the proof.  
\end{proof}

\section{Concluding Remarks}
\begin{itemize}
\item The lone case of $r=2, n=5$ has not been considered in the preceding sections. In this case, in fact, $D(K(5,2))=3$. It is a simple calculation (again using a randomized coloring) to show that in this case too, $D_l(K(5,2))=3$. We omit the (simple) details.
\item While we were content with showing that the with positive probability, a random list-coloring of the vertices of $K(n,r)$ (for $r\ge 3$) actually is distinguishing, it is easy to see that in fact, these are asymptotically almost sure events, like in the case of $r=2$. In particular, these give very efficient randomized algorithms for distinguishing list colorings.
\item Our methods may possibly also extend to yield other results of the same kind. An instructive instance would be to consider an $r$-fold cartesian product of complete graphs; the distinguishing number of  cartesian products of complete graphs was shown to be $2$ in \cite{KX} though it is not yet known if the list distinguishing number also equals $2$, and we believe that the same ideas may turn out to be useful there (though the computations can get more complicated).
\item As observed in remark \ref{probzero}, the expressions for the probabilities as calculated in most sections are non-zero only if certain lists are identical, otherwise the probabilities are in fact much lower. We believe that the following strengthening of the List Distinguishing Conjecture is also true:
\begin{conj} For a graph $G$, with a collection of equal sized (size $k$, say) lists $\List=\{L_v|v\in V\}$, if $p(\List)$ denotes the probability that a random coloring (obtained by choosing for each vertex, a color from its list uniformly and independently) admits a non-trivial automorphism which preserves all the color classes, then $p(\List)$ is maximized when the lists are identical.
\end{conj}
Our results, while not quite proving this stronger statement exactly (since computing these probabilities exactly would be cumbersome) in fact proves that the expected number of non-trivial automorphisms that fix all the color classes is actually maximized when the lists are identical. 
\end{itemize}

\section*{Appendix: $D_l(K(n,2))=2$ for $n=6,7$}
Consider a graph $G$ with a collection of lists $\List=\{L(e)|e\in E(G)\}$.   By  the \textbf{color palette of a vertex $v$} in a graph $G$, we mean the multi-set of colors assigned to the incident edges of $v$ in a list coloring of the edges of $G$. A monochromatic path $P$ shall refer to a maximum sized path in $G$ such that $\mathop{\cap}\limits_{e\in E(P)} L(e) \neq \emptyset$. We use $l(P)$ to denote the length of $P$ and $|P|$ to denote the number of vertices in $P$. Hereafter the word path shall also refer only to monochromatic paths.

\begin{lem}\label{random} Let $n\ge 6$ Suppose we have a collection of lists $\List=\{L(e)|e\in E(K_n\}$ of size $2$. If there is no monochrome path in $K_n$ of length two then there is a distinguishing list coloring of the edges of $K_n$ from the lists in $\List$.
\end{lem}  
\begin{proof}
For each edge $e \in E(K_n)$,  choose a random coloring of the edges as before.  Observe that for any color fixing automorphism $\sigma$, the color pallettes of $u$ and $\sigma(u)$ are the same. More over the edges $uv$ and $\sigma(u)\sigma(v)$ have the same color. The probability that there exist vertices $u,v$ with the same color pallettes is at most $\binom{n}{2}/2^{n-2}$ since for any color incident at vertex $u$, and not on the edge $uv$, there is at most one edge incident with $v$ that can have that color in its list, by the hypothesis. Now, for $\binom{n}{2}/2^{n-2}<1$ for $n\ge 6$, so we are through.
\end{proof}
By the virtue of lemma \ref{random}, we may assume that $K_n$ contains a monochromatic path of length at least two.  We introduce some further terminology. $K_n$ shall be the complete graph on the vertex set $[n]$, and  $G'$ we shall denote the complete subgraph on $[n]\setminus V(P)$. The edges between $G'$ and $P$ will be referred to as crossing edges. $e_{ij}$ is the edge between vertex $i$ and $j$ and $c_{ij}$ shall denote the color assigned to the edge $e_{ij}.$ The available common color on the edges of $P$ is denoted $c_1.$ Without loss of generality we assume $V(P) = \{ 1, 2, \hdots, |P|\}.$

\begin{thm}\label{n=6,2}
$D_l(K(6,2))= 2$.
\end{thm}
 \begin{proof}
As observed before, we may assume that  if $P$ is a monochromatic path, then $|P| \ge 3$. 
Consider the following cases 
\begin{itemize}
\item[1.] $|P| =6:$
Color $E(P)$ using $c_1,$ avoid $c_1$ from all other edges except $e_{24}$ and $e_{35}$. Also ensure that $c_{24} \neq c_{35}$. This coloring is distinguishing since the color class $c_1$ is fixed (as a set) only by two maps - the identity and the permutation $\sigma=(16)(25)(34)$. But since $\sigma(e_{24}) = e_{35}$, and they are colored differently,  $\sigma$ does not fix every color class.
\item[2.]$|P|=5$: Assign $c_1$ to $E(P)$ and avoid $c_1$ from all other edges. Again, ensure that $c_{16} \neq c_{56}$; $G'$ consists of the lone vertex $6$ and $|P|=5$, so $c_1$ does not appear on the lists of  both $e_{16}$ and $e_{56}$, so this arrangement is possible.  By our choices, no crossing edge is colored $c_1$, so the monochrome set of edges colored $c_1$ is again precisely $P$. This coloring is distinguishing for very similar reasons as above. 
\item[3.]$|P|= 4$:
Assign $c_1$ to $E(P)$ and avoid $c_1$ from all other edges. Ensure that $c_{45} \neq c_{14}, c_{45} \neq c_{16}$, and $c_{45} \neq c_{46}$; again, these arrangements are possible by the maximality of $P$ as none of the crossing edges from the end vertices of $P$ contain $c_1$ in the given lists. It is now easy to check that this coloring is distinguishing.
 \item[4.] $|P|= 3$: We start by coloring the edges of $P$ using $c_1$. Color the edges $e_{16}$ and $e_{46}$ arbitrarily from their lists, and for the remaining edges, impose a restriction on the color that needs to be assigned to it as in  Table \ref{colortable6.1} below.  Again, note that the maximality of $P$ ensures that all these avoidances are permissible.
 \end{itemize}
   \begin{table}[h]
\begin{center}
\begin{tabular}{| c | c | c |}
  \hline                       
  Edges  & Restriction on the color choice\\
  \hline \hline
 $e_{12}, e_{23} $  & Assign $c_1$\\
  \hline 
   $e_{24},e_{25},e_{26},e_{13},e_{45}$ & Avoid $c_1$\\
    \hline
  $e_{34},e_{35},e_{36},e_{14},e_{15}$ & Avoid $c_{16}$\\
  \hline
    $e_{56}$  & Avoid $c_{46}$ \\
  \hline
\end{tabular}
\caption{Coloring Scheme when $n=6$}\label{coloring table} \label{colortable6.1}
\end{center}
\end{table}

To see why this is distinguishing, suppose $\sigma$ is an automorphism that fixes each of these color classes. By the avoidance choices, the only edges that are colored $c_1$ are the edges of $P$. Our choices also ensure that the pallettes of vertices $1$ and $3$ are different, so it follows that $\sigma $ fixes $1,2,3$. Since $c_{46}\ne c_{56}$, $\sigma\neq (45), (456),(465)$ and since $c_{14}, c_{15}\ne c_{16}$, $\sigma\neq (46), (56)$, so $\sigma$ is the identity map on $[6]$. 
\end{proof}
\begin{thm}
$D_l(K(7,2)) = 2.$
\end{thm}

\begin{proof} 
We proceed as we did in the theorem \ref{n=6,2} and consider the following cases.
\begin{itemize}
\item[1.] When $|P| \ge 5$ the coloring scheme is similar to that of $|P| \ge 4$ in theorem \ref{n=6,2}. If $|P| = 7,$ assign $c_1$ to all the edges of $E(P),$ ensure $c_{24} \neq c_{46}$ and avoid $c_1$ from all other edges. For $|P| = 6,$ assign $c_1$ to $E(P),$ ensure $c_{24} \neq c_{35}$ and avoid $c_1$ from all other edges. For $|P| = 5,$ assign $c_1$ to $E(P),$ ensure $c_{56} \notin \{c_{16}, c_{17}, c_{57} \}$ and avoid $c_1$ from all other edges. The proofs that these give distinguishing colorings is similar to the arguments that appear in theorem \ref{n=6,2}, so we omit those details.
\item[2.] $|P| = 4:$  Assign $c_1$ to all the edges of $P$, and ensure that $c_{56}\neq c_1$. Also, avoid $c_{56}$ from $e_{67}$ and $e_{57}.$ Further ensure $c_{17} \neq c_{47}$ and  $c_{16}\ne c_{15}$ from all other edges. As always, avoid $c_1$ on any crossing edge.\\
Our choice of coloring guarantees that any automorphism $\sigma$ that fixes every color class necessarily maps the set $\{1,2,3,4\}$ and $\{5,6,7\}$ into themselves respectively. Since $c_{57}, c_{67}\ne c_{56}$, $\sigma(7)=7$ and since $c_{17} \neq c_{47}$ it follows that $\sigma$ fices each of $1,2,3,4$. Finally, since $c_{16}\ne c_{15}$, $\sigma$ fixes $5,6$ as well.
\item[3.] $|P|= 3$: Color the edges of $P$ using $c_1$. Color the edges $e_{16}$ and $e_{46}$ arbitrarily from their lists, and for the remaining edges, we consider two sub cases and in each sub case we impose a different type of restriction on the color that needs to be assigned to the edges; see  Tables \ref{colortable7.1} and \ref{colortable7.2} for the details on the restrictions.  Again, note that the maximality of $P$ ensures that all these avoidances are permissible. 
\end{itemize}
\begin{table}[h]
{Sub case 1.} $c_1 \in L(e_{27}).$ 
\begin{center}
\begin{tabular}{| c | c | c |}
  \hline                       
  Edges  & Restriction on the color choice\\
  \hline \hline
 $e_{12}, e_{23} $  & Assign $c_1$\\
  \hline 
   $e_{24},e_{25},e_{26},e_{13},e_{45},e_{27}$ & Avoid $c_1$\\
    \hline
  $e_{34},e_{35},e_{36},e_{14},e_{15}, e_{37}$ & Avoid $c_{16}$\\
  \hline
    $e_{47},e_{56}, e_{67}$  & Avoid $c_{46}$ \\
  \hline
 $e_{17}$  & Avoid $c_{15}$ \\
  \hline
   $e_{37}$  & Avoid $c_{36}$ \\
  \hline
\end{tabular}
\caption{Coloring Scheme when $n=7$}\label{colortable7.1}
\end{center}
\end{table}
\begin{table}[h]
{Sub case 2.} $c_1 \notin L(e_{27}).$ 
\begin{center}
\begin{tabular}{| c | c | c |}
  \hline                       
  Edges  & Restriction on the color choice\\
  \hline \hline
 $e_{12}, e_{23} $  & Assign $c_1$\\
  \hline 
   $e_{24},e_{25},e_{26},e_{13},e_{45},e_{47},e_{57},e_{67}$ & Avoid $c_1$\\
    \hline
  $e_{34},e_{35},e_{36},e_{14},e_{15}, e_{37}$ & Avoid $c_{16}$\\
  \hline
    $e_{56}$  & Avoid $c_{46}$ \\
  \hline
 $e_{17}$  & Avoid $c_{15}$ \\
  \hline
  $e_{27}$  & Avoid $c_{24}$ \\
  \hline
   $e_{37}$  & Avoid $c_{36}$ \\
  \hline
\end{tabular}
\caption{Coloring Scheme when $n=7$}\label{colortable7.2}
\end{center}
\end{table}
In sub case 1, Firstly we observe that by our choices, we in fact have $c_{37}\neq c_{36}$ because by the hypothesis of sub case 1, $L(e_{36})$ and $L(e_{37})$ cannot both have the color $c_{16}$, otherwise $|P|\ge 4$.  Further, the hypothesis of sub case 1 guarantees that $c_1$ is not present in the lists of $e_{47}$ and $e_{67}$, so our avoidances in fact give us that $P$ is the unique path of length $2$ colored $c_1$. Since the pallettes of $1$ and $3$ are different, it follows that any $\sigma$ that preserves all the color classes must necessarily fix $1,2,3$. Now we first show that $7$ is also fixed. Indeed, if $\sigma(6)=7$, then $\sigma(e_{36})=e_{37}$ but by choice, these are colored differently. Similarly, $\sigma(5)\neq 7$ since $c_{15}\neq c_{17}$. Now, if $\sigma(4)=7$, then $\sigma(6)=5$ as a consequence of our color avoidances. But then $c_{15}\ne c_{16}$, so this shows that $\sigma$ fixes $7$ as well. Finally, by following similar arguments as in theorem \ref{n=6,2}, it follows that $\sigma$ fixes $4,5,6$ as well, so $\sigma $ is the identity map.

In sub case 2, the crucial difference is in the color choice of $e_{27}$. The color avoidance here ensures that $\sigma(4)=7$ or $\sigma(7)=4$ is not possible since $c_{24}\ne c_{27}$. The rest of the proof is similar to sub case 1.
\end{proof}
\end{document}